\documentclass{amsart}
\usepackage{epsfig}
\usepackage{amscd}
\usepackage{amsmath}
\usepackage{amsxtra}
\usepackage{amsfonts}
\usepackage{amssymb}
\newtheorem{theorem}{Theorem}[section]

\newtheorem{lemma}[theorem]{Lemma}
\newtheorem{proposition}[theorem]{Proposition}
\theoremstyle{definition}
\newtheorem{definition}[theorem]{Definition}
\newtheorem{remark}[theorem]{Remark}

\newtheorem{example}[theorem]{Example}
\theoremstyle{remark}

\renewcommand{\theclaim}{\textup{\theclaim}}

\newtheorem*{acknowledgements}{Acknowledgments}

\numberwithin{equation}{section}

\def\openone%{\hbox{\upshape \small1\kern-3.3pt\normalsize1}}
{\mathchoice
{\hbox{\upshape \small1\kern-3.3pt\normalsize1}}
{\hbox{\upshape \small1\kern-3.3pt\normalsize1}}
{\hbox{\upshape \tiny1\kern-2.3pt\SMALL1}}
{\hbox{\upshape \Tiny1\kern-2pt\tiny1}}}
\makeatletter
\newbox\ipbox
\newcommand{\ip}[2]{\left\langle #1\mathrel{\mathchoice
{\setbox\ipbox=\hbox{$\displaystyle \left\langle\mathstrut #1#2\right\rangle$}
\vrule height\ht\ipbox width0.25pt depth\dp\ipbox}
{\setbox\ipbox=\hbox{$\textstyle \left\langle\mathstrut #1#2\right\rangle$}
\vrule height\ht\ipbox width0.25pt depth\dp\ipbox}
{\setbox\ipbox=\hbox{$\scriptstyle \left\langle\mathstrut #1#2\right\rangle$}
\vrule height\ht\ipbox width0.25pt depth\dp\ipbox}
{\setbox\ipbox=\hbox{$\scriptscriptstyle \left\langle\mathstrut #1#2\right\rangle$}
\vrule height\ht\ipbox width0.25pt depth\dp\ipbox}
} #2\right\rangle}
\newcommand{\diracb}[1]{\left\langle #1\mathrel{\mathchoice
{\setbox\ipbox=\hbox{$\displaystyle \left\langle\mathstrut #1\right.$}
\vrule height\ht\ipbox width0.25pt depth\dp\ipbox}
{\setbox\ipbox=\hbox{$\textstyle \left\langle\mathstrut #1\right.$}
\vrule height\ht\ipbox width0.25pt depth\dp\ipbox}
{\setbox\ipbox=\hbox{$\scriptstyle \left\langle\mathstrut #1\right.$}
\vrule height\ht\ipbox width0.25pt depth\dp\ipbox}
{\setbox\ipbox=\hbox{$\scriptscriptstyle \left\langle\mathstrut #1\right.$}
\vrule height\ht\ipbox width0.25pt depth\dp\ipbox}
}\right. }
\newcommand{\dirack}[1]{\left. \mathrel{\mathchoice
{\setbox\ipbox=\hbox{$\displaystyle \left.\mathstrut #1\right\rangle$}
\vrule height\ht\ipbox width0.25pt depth\dp\ipbox}
{\setbox\ipbox=\hbox{$\textstyle \left.\mathstrut #1\right\rangle$}
\vrule height\ht\ipbox width0.25pt depth\dp\ipbox}
{\setbox\ipbox=\hbox{$\scriptstyle \left.\mathstrut #1\right\rangle$}
\vrule height\ht\ipbox width0.25pt depth\dp\ipbox}
{\setbox\ipbox=\hbox{$\scriptscriptstyle \left.\mathstrut #1\right\rangle$}
\vrule height\ht\ipbox width0.25pt depth\dp\ipbox}
} #1\right\rangle}

\input{psfig.sty}
\input cyracc.def

\begin{document}
\title[New Presentations of Thompson's Groups]{New Presentations of Thompson's Groups and Applications }
\author{Uffe Haagerup}
\address{Department of Mathematics and Computer Science\\
University of Southern Denmark\\
Campusvej 55\\
DK-5230 Odense M\\
Denmark}
\email{haagerup@imada.sdu.dk}
\author{Gabriel Picioroaga}
\address{Department of Mathematical Sciences\\
Binghamton University\\
U.S.A.}
\email{gabriel@math.binghamton.edu}
\thanks{}

\begin{abstract} We find new presentations for the Thompson's groups $F$, the derived group $F^{'}$ and the intermediate group $D$. These presentations have a common ground in that their relators are the same and only the generating sets differ. As an application of these presentations we extract the following consequences: the cost of the group $F^{'}$ is $1$ hence the cost cannot decide the (non)amenability question of $F$; the $II_1$ factor $L(F^{'})$ is inner asymptotically abelian and the reduced $C^*$-algebra of $F$ is not residually finite dimensional.
\end{abstract}\maketitle

\section{Introduction}
The Thompson group $F$ can be regarded as the group of piecewise-linear,\\ orientation-preserving 
homeomorphisms of the unit interval which have breakpoints only at dyadic points and on intervals of 
differentiability the slopes are powers of two. The group was discovered in the '60s by Richard Thompson 
and in connection with the now celebrated groups $T$ and $V$ led to the first example of a finitely 
presented infinite simple group. Also, it has been shown that the commutator subgroup $F^{'}$ of $F$ is simple. 
\par In 1979 R. Geoghegan conjectured that $F$ is not amenable. This problem is still open and of importance for group theory: either outcome will help better understand the inclusions $\mathcal E A\subset\mathcal A G\subset\mathcal N F$, where $ \mathcal E A$ is the class of elementary amenable groups, $ \mathcal A G$ is the class of amenable groups and $\mathcal N F$ is the class of groups not containing free (non-abelian) groups. By work of Grigorchuck \cite{Gr}, Olshanskii and Sapir \cite{OS}, the inclusions above are strict.
\par There are properties stronger than amenability and also weaker ones. There is naturally a great deal of interest in knowing which ones hold or fail in the case of $F$. For example, from the 'weak'  perspective the question of exactness has been put forward in \cite{AGS}. We also find a two-folded interest in whether or not the reduced $C^*$ algebra of $F$ is quasidiagonal (QD): by a result of Rosenberg in \cite {Ha} this property implies that the group is amenable. It is also conjectured that any countable amenable group generates a QD reduced $C^*$ algebra. As a consequence, the (non)QD property gives another spin to the amenability question of $F$. We prove a weaker result than non-QD, namely the reduced $C^*$ algebra of $F$ is not {\it{residually finite dimensional}}. 
\par The Thompson's groups have infinite conjugacy classes and therefore the associated von Neumann algebras are $II_1$ factors (see \cite{Jol}). Also, P. Jolissaint proved that the $II_1$ factor associated with the Thompson group $F$ has the relative McDuff property with respect to the $II_1$ factor determined by $F^{'}$; in particular both are McDuff factors. By finding a presentation of the commutator subgroup we naturally recover another result of Jolissaint (\cite{JR}), namely that the $II_1$ factor $L(F^{'})$ is (inner) asymptotically abelian. The last property has been introduced by S. Sakai in the 70's and consists essentially of a stronger requirement than property $\Gamma$ of Murray and von Neumann: instead of a sequence of unitaries almost commuting with the elements of the factor one wants a sequence of (inner) *-isomorphisms to do the job. Moreover, asymptotically abelian is a stronger property than McDuff (see the Background section below). 
\par In \cite{Gab}, D. Gaboriau introduced a new dynamical invariant for a countable discrete group called  cost. Infinite amenable groups have cost 1 and also Thompson's group $F$ has cost 1, while the free group on $n$ generators has cost $n$. As the cost non-decreases when passing to normal subgroups, finding the cost of $F^{'}$ becomes an interesting question. Using our new presentation of $F^{'}$ and one of the tools developed by Gaboriau we show that $F^{'}$ has cost 1 as well (and because $F^{'}$ is simple we get that any non-trivial normal subgroup of $F$ has cost 1). It is very likely that any non-trivial subgroup of $F$ has cost $1$. This problem might be related to a conjecture of M. Brin: {\it any subgroup of $F$ is either elementary amenable or contains a copy of $F$} (Conjecture 4 in \cite{Br}). 
\par The paper is organized as follows: the Background  section prepares some basics on the Thompson groups, group von Neumann algebras and cost of groups. We have collected some known facts and also folklore-like facts, mostly about the Thompson's groups. The follow-up to this section is our main result which describes various presentations of the groups $F$, $F^{'}$ and $D$. Next section of the paper contains conclusions of these presentations.

\section{Background}
\subsection{Thompson's Groups} For a good introduction of Thompson's groups we refer the reader to \cite{Can}.  
\begin{definition}\label{def1} The Thompson group $F$ is the set of piecewise 
linear homeomorphisms from the closed unit interval $[0,1]$ to itself that are differentiable except at 
finitely many dyadic rationals and such that on intervals of differentiability the derivatives are 
powers of $2$.

\end{definition}
\begin{remark}\label{fgen} The group $F$ is shown to have the following finite presentation: \\
$\left< A, B\right>$ with relations $[AB^{-1},A^{-1}BA]=1$ and $[AB^{-1},A^{-2}BA^2]=1$. \\
Also, $F$ has a useful infinite presentation: \\ 
$F=\left< x_0, x_1, ...x_i,...|\mbox{ }x_jx_i=x_ix_{j+1}\mbox{, }i<j\mbox{ }\right>$.\\
This is obtained by declaring $x_0=A$, $x_n=A^{-(n-1)}BA^{n-1}$. 
\end{remark}
\begin{figure}\label{f0}
\begin{center}
\input{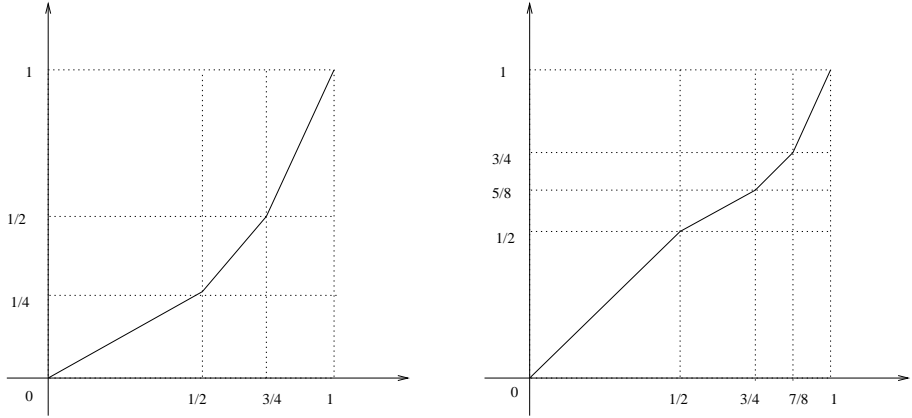}
\caption{Graphs of generators $A=x_0$ and $B=x_1$ }
\end{center}
\end{figure}

As a map on the unit interval $x_n$ is given by
\begin{equation}\label{xn}
x_n(t)=\left\{
\begin{array}{lr}
t,&\mbox{ if }0\leq t\leq 1-2^{-n} \\
\frac{t}{2}+\frac{1}{2}(1-2^{-n}),&\mbox{ if }1-2^{-n}\leq t\leq 1-2^{-n-1}\\
t-2^{-n-2},&\mbox{ if }1-2^{-n-1} \leq t\leq 1-2^{-n-2}\\
2t-1,&\mbox{ if }1-2^{-n-2}\leq t\leq 1\\
\end{array}\right.
\end{equation}
The following result can be found in \cite{Can}. 
\begin{proposition}\label{pro0} Let $F$ be given as in definition \ref{def1}.\\ 
i) The subgroup  
 $$F^{'}:=\{f\in F\mbox{ }|\mbox{ }\exists\mbox{}\delta,\epsilon\in (0,1)\mbox{ such that } f_{|[0,\epsilon]}=id\mbox{, }f_{|[\delta,1]}=id\mbox{ }\}$$ 
is normal and simple. Moreover, $F^{'}$ is the commutator (or the derived group) of $F$.\\
ii) Any non-trivial quotient of $F$ is abelian.
\end{proposition}
As a consequence, any non trivial normal subgroup of $F$ must contain $F^{'}$. In the next section we will give a (infinite) presentation of $F^{'}$ and of the intermediate normal subgroup introduced in \cite{Jol} 
$$D:=\{f\in F\mbox{ }|\mbox{ }\exists\mbox{}\delta\in (0,1)\mbox{ such that }f_{|[\delta,1]}=id,\mbox{ }\}.$$
The following finite presentation of $F$ is well known to specialists. Starting with $n\geq 4$,  with notations $A=x_0$, $B=x_1$ we can follow the proof of Theorem 3.1 in \cite{Can} by rewritting first the two relators in the finite presentation of $F$ as $x_3=x_1^{-1}x_2x_1$ and $x_4=x_1^{-1}x_3x_1$. 

\begin{lemma}\label{lem1} Let $n\geq 4$. The Thompson group $F$ is isomorphic to the group generated by $x_0$, $x_1$,..., $x_n$ subject to relations 

\begin{equation}\label{fig}
x_jx_i=x_ix_{j+1}\mbox{ for all } 0\leq i< j\leq n-1
\end{equation}
(Only $n+1$ generators are used.)
\end{lemma}
\begin{remark}\label{remR}
There is a classic procedure to realize $F$ as a group of transformations on $\mathbb{R}$. Let \~{F} be a subgroup of the group of piece-wise linear transformations of the real line such that its elements:\\
$\bullet$ have finitely many breakpoints and only at dyadic real numbers;\\
$\bullet$ have slopes in $2^{\mathbb{Z}}$;\\
$\bullet$ are translations by integers outside a dyadic interval.\\
Then \~{F}$=\varphi F\varphi^{-1}$ where $\varphi:(0,1)\rightarrow\mathbb{R}$ is defined as follows:\\
 $\varphi(t_n)=n$ and $\varphi$ is affine in $[t_n,t_{n+1}]$, for all $n\in\mathbb{Z}$ where 
$$t_n=\left\{
\begin{array}{lr}
1-(\frac{1}{2})^{n+1},&\mbox{ if }n\geq 0 \\
(\frac{1}{2})^{1-n},&\mbox{ if }n<0 
\end{array}\right.$$
To recover the generators in this new setting notice that $x_0(t_n)=t_{n-1}$. The corresponding generator of \~{F} thus satisfies \~{x}$_0(t)=t-1$ for all $t\in\mathbb{R}$. Also, by definition for $n\geq 1$,   $x_n=x_0^{-(n-1)}x_1x_0^{n-1}$ which together with the action of $x_0$ on the sequence $(t_m)_{m\in\mathbb{Z}}$ determines the form of the other generators, for $n\geq 1$: \~{x}$_n(t)=t$ for all $t\leq n-1$, \~{x}$_n(t)=\frac{t+n-1}{2}$ for $n-1\leq  t\leq n+1$, \~{x}$_n(t)=t-1$, for all $t\geq n+1$, (see figure 2). In conclusion the group \~{F} is generated by $(\mbox{\~{x}}_n)_{n\in\mathbb{N}}$ and the similar $F$ relations from the infinite presentation of $F$ constitute a presentation of \~{F}. We will see later that it is useful to consider maps  \~{x}$_n$ with negative integers $n$. Our aim is to  give a presentation of the commutator subgroup of $F$, thus it suffices to find a presentation of the commutator subgroup of \~{F}. Using the description of the commutator in Proposition \ref{pro0} we obtain that the commutator in \~{F} is 

\begin{equation}\label{commR}
\mbox{\~{F}}\mbox{}^{'}=\{f\in\mbox{ \~{F}  }|\mbox{ }f(t)=t\mbox{ , }|t|\geq k\mbox{ for some }k\in\mathbb{N}\}
\end{equation}

\end{remark}

\begin{figure}\label{f1}
\begin{center}
\input{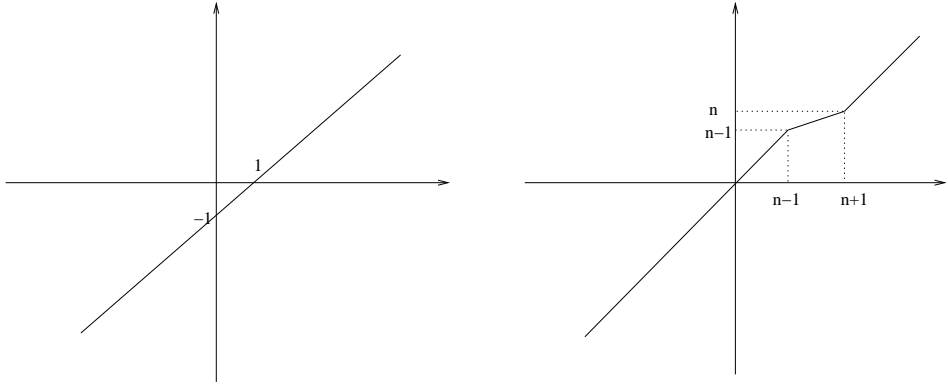}
\caption{Graphs of generators \~{x}$_0$ and \~{x}$_n$, $n\geq 1$ }
\end{center}
\end{figure}

\subsection{Group von Neumann Algebras}
\par If $G$ is a countable discrete group with infinite conjugacy classes (i.c.c.) then the left 
regular representation of $G$ on $l^2(G)$ gives rise to a $II_1$ factor, the group von Neumann algebra $\mathcal{L}(G)$, 
as follows: endow $l^2(G)=\left\{ \psi:G\rightarrow\mathbb{C}\mbox{ }|\mbox{}\sum_{g\in G}|\psi(g)|^2<\infty\right\}$  with the scalar product 
$$\left<\phi,\psi\right>:=\sum_{g\in G}\phi(g)\overline{\psi(g)}$$
The Hilbert space $l^2(G)$ is generated by the countable collection of vectors 
$\left\{\delta_g\mbox{ }|\mbox{}g\in G\right\}$. Also, an element $g\in G$ defines a unitary operator $L_g$, 
on $l^2(G)$  as follows: 
$L_g(\psi)(h)=\psi(g^{-1}h)$, for any $\psi\in l^2(G)$ and any $h\in G$. (Sometimes, to not burden 
the notation we will write just $g$ instead of $L_g$). Now, $\mathcal{L}(G)$, 
the von Neumann algebra generated by $G$ is obtained by taking the wo-closure in $B(l^2(G))$ (all bounded operators on $l^2(G)$), of the linear span of the set 
$\left\{L_g \mbox{ }|\mbox{}g\in G\right\}$ $($if one takes the norm closure of the same linear span then one obtains the reduced $C^*$ algebra of the group, $C^*_r(G)$ $)$. It is well known (see \cite{KR}) that $\mathcal{L}(G)$ is a factor provided $G$ is an i.c.c. group (i.e. every conjugacy class in $G\setminus\{e\}$ is infinite) and it is of type $II_1$. The map defined  by 
$\mbox{tr}(x)=\left<x(\delta_e),\delta_e\right>$, where $e\in G$ is the neutral element and $x\in\mathcal{L}(G)$  
is a faithful, finite, normal trace. The canonical trace also determines the Hilbertian norm $||x||_2=\mbox{tr}(x^{*}x)^{1/2}$. In particular, for $x$, $y$ in $\mathcal L (G)$ the following inequalities hold: $||xy||_2\leq ||x||\mbox{ }||y||_2$ and $||yx||_2\leq ||y||_2\mbox{ }||x||$, where $||x||$ is the usual operator norm of $x$ in $B(l^2(G))$. Also, if $u$ is unitary in $\mathcal{L}(G)$ then $||xu||_2=||ux||_2=||x||_2$ for any $x\in\mathcal{L}(G)$.\\
Finally let us recall an important result of A. Connes (see \cite{Co}): If $G$ is a countable i.c.c.  group, $\mathcal{L}(G)$ is the hyperfinite $II_1$ factor if and only if $G$ is amenable.

\begin{definition} A finite factor $M$ is called asymptotically abelian if there exists a sequence of *-automorphisms $(\rho_n)_{n\in\mathbf N}$ on $M$ such that 
$$||[\rho_n(a),b]||_2\rightarrow 0\mbox{ for }a,b\in M.$$
If each $\rho_n$ is inner then $M$ is called inner asymptotically abelian. 
\end{definition}

\begin{example} (see \cite{Sa})
\begin{itemize}
\item The type $I_n$ factor is not asymptotically abelian. 
\item The hyperfinite $II_1$ factor $\mathcal R$ is asymptotically abelian.
\item Any asymptotically abelian factor is McDuff (this follows from characterization of McDuff property with central sequences).
\item $\mathcal{L}(F_2)\otimes\mathcal R$ is not asymptotically abelian and is a McDuff factor ( $F_2$ is the free group on two generators). 
\item If $M$ is a finite factor then $\otimes_{i=1}^{\infty}M$ is asymptotically abelian. 
\end{itemize}
\end{example}

\subsection{Cost of Groups}
We collect here definitions and some results from \cite{Gab}. 
We say that  $R$ is a SP1 equivalence relation on a standard Borel probability space $(X,\lambda)$ if
\par (S) Almost each orbit $R[x]$ is at most countable and $R$ is a Borel subset of $X\times X$.
\par (P) For any $T\in\mbox{Aut}(X,\lambda)$ such that $\mbox{graph}T\subset R$ we have that $T$ preserves the measure $\lambda$.
\begin{definition}i) A countable family $\Phi=(\varphi_i:A_i\rightarrow B_i)_{i\in I}$
of measure preserving, Borel partial isomorphisms between Borel subsets of $(X,\lambda)$ is called a
graphing on $(X,\lambda)$.\\
ii) The equivalence relation $R_{\Phi}$ generated by a graphing $\Phi$ is the smallest equivalence relation $S$ such that
$(x,y)\in S$ iff $x$ is in some $A_i$ and $\varphi_i(x)=y$. \\
iii) An equivalence relation $R$ is called treeable if there is a graphing $\Phi$ such that $R=R_{\Phi}$ and almost
every orbit $R_{\Phi}[x]$ has a tree structure. In such case $\Phi$ is called a treeing of $R$.
\end{definition}
One can consider the quantity $C(\Phi)=\sum\lambda(A_i)$. The cost
of a (SP1) equivalence relation is defined by the number 
$$C(R):=\inf\{C(\Phi) |  \Phi\mbox{ is a graphing of } R\}$$
It is the preserving property that allows one to conclude the infimum is attained iff $R$ admits a treeing (see Prop.I.11 and Thm.IV.1 in
\cite{Gab}). The numbers $C(R)$ could be interpreted as the "cheapest" measure-theoretical way to generate $R$ with partial isomorphisms on standard probability space $(X,\lambda)$. The cost of a discrete countable group $G$ is  
$$C(G):=\inf\{C(R) |  R \mbox{ coming from a free, measure preserving
action of }G\mbox{ on }X \}$$
If all numbers $C(R)$ are equal then the group is said to be of fixed price. The cost does not depend on the standard Borel probability space $(X,\lambda)$ as all standard Borel spaces are isomorphic as measure spaces. \\
The following statements were proved by Gaboriau. 
\begin{theorem}\label{tga} \cite{Gab}\\
1) The cost of an infinite, amenable group is $1$, fixed price.\\
2) The Thompson group $F$ has cost $1$, fixed price. \\
3) The cost of the free group on $n$ generators is $n$, fixed price. \\
4) If $N$ is a infinite normal subgroup of $G$, of fixed price then\\ $C(N)\geq C(G)\geq 1$.\\
5) Any number $c\geq 1$ is the cost (fixed price) of some group.\\
6)If $G$ is an increasing union of infinite groups $(G_n)_n$ such that $C(G_1)=1$, fixed price and if $G_{n+1}$ is generated by $G_n$ and elements $\gamma\in G$ such that  $\gamma^{-1}G_n\gamma\cap G_n$ is infinite then $G$ is of cost $1$, fixed price.
\end{theorem}

\section{Main result}

Recall the following general principle (von Dyck): let $G=\left< X\mbox{}|\mbox{}\mathcal R\right>$ be a group generated by a set $X$ subject to the set of relators  $\mathcal R$. Let $F(X)$ be the free group on $X$ generators, $H$ is an arbitrary group and $f:X\rightarrow H$ a function. Denote by ${\it v}$ its morphism extension to $F(X)$. If ${\it v}(\mathcal R)=1$ in $H$ then the map $f$ can be extended to a morphism from $G$ to $H$. Moreover, if $f(X)$ generates $H$ then this morphism is surjective. 

\par Before stating the main result we will make some preparations. These will be fully used in the second part of the proof below.  Let us turn to the point of view taken in Remark \ref{remR}. Recall that we can work with the group \~{F} and its generators \~{x}$_n$'s instead of $F$ and $x_n$'s. Moreover, same relations as in remark \ref{fgen} hold in (and present) \~{F}. First we will extend the sequence $(x_n)_n$ for negative values of $n$. Define a sequence of elements in $F$ as follows: \\
$\overline{x_n}:=x_n$ for $n\geq 1$.\\ 
$\overline{x_0}:=x_0x_1x_0^{-1}$.\\ 
$\overline{x_n}:=x_0^{-(n-1)}x_1x_0^{n-1}$ for $n<0$.\\
From the above we get $\overline{x_n}=x_0^{-n}\overline{x_0}x_0^n$ which entails 
$$\overline{x_{n+1}}=x_0^{-1}\overline{x_n}x_0\mbox{, for all }n\in\mathbb{Z}$$
By $y_n$ we will denote the image of $\overline{x_n}$ in \~{F} (see figure 3). 
We have 
\begin{equation}\label{Yrec}
y_{n+1}=\mbox{\~{x}}_0^{-1}y_n\mbox{\~{x}}_0 
\end{equation}
Notice that from the relations of type $\mbox{\~{x}}_j\mbox{\~{x}}_i=\mbox{\~{x}}_i\mbox{\~{x}}_{j+1}$ we obtain by translation
\begin{equation}\label{Ys}
y_jy_i=y_iy_{j+1}\mbox{  for any }i<j\mbox{ ,  }i\mbox{, }j\mbox{  in }\mathbb{Z}
\end{equation}
(The 'obvious' extension $\overline{x}_0=x_0$ would have destroyed (\ref{Ys}) ,e.g. pick  $i=-1$ and $j=0$.)
For $i\in\mathbb{Z}$ we define now the maps   
\~{G}$_i:\mathbb{R}\rightarrow\mathbb{R}$ by $\mbox{\~{G}}_i=y_iy_{i+1}^{-1}$ (see figure 4). 
For example:
$$\mbox{\~{G}}_0(t)=\left\{
\begin{array}{lr}
t,&\mbox{ if }t\leq -1 \\
\frac{t-1}{2},&\mbox{ if }-1\leq t\leq 0 \\
t-\frac{1}{2},&\mbox{ if }0\leq t\leq \frac{1}{2}\\
2t-1,&\mbox{ if }\frac{1}{2}\leq t\leq 1\\
t,&\mbox{ if }1\leq t
\end{array}\right.$$
By (\ref{commR}) we get that each \~{G}$_i$ belongs to the commutator \~{F}$^{'}$.
% Also, with the aid of (\ref{Yrec}) we obtain the following identity:
%\begin{equation}\label{Grec}
%\mbox{\~{G}}_{i+1}=\mbox{\~{x}}_0^{-1}\mbox{\~{G}}_i\mbox{\~{x}}_0=\mbox{\~{G}}_0(\cdot -i)+i
%\end{equation}

\begin{figure}\label{f2}
\begin{center}
\input{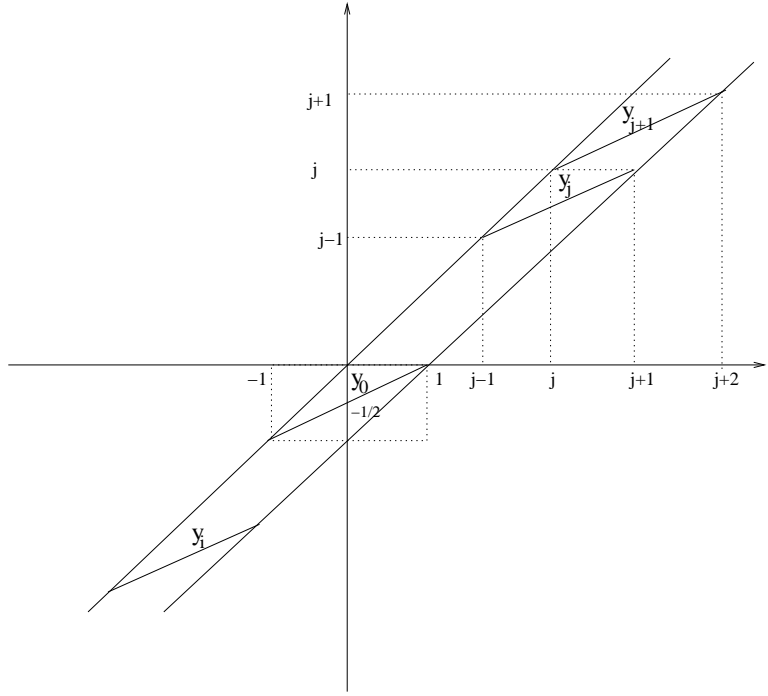}
\caption{Graphs of $y_i$, $i\in\mathbb{Z}$. Note $y_jy_i=y_iy_{j+1}$ when $i<j$.} 

\end{center}

\end{figure}
We are now ready to prove the main result of the paper: 
\begin{theorem}\label{teo}Let $I\subset \mathbf Z$ be a set of consecutive integers and $G$ the group generated (and presented) by $(g_i)_{i\in\mathbf I}$ subject to relations:

\begin{equation}\label{r1}
g_{i-1}g_ig_{i+1}=g_ig_{i+1}g_{i-1}g_i
\end{equation}
\begin{equation}\label{r2}
[g_i,g_j]=1\mbox{ ,   }|i-j|\geq 2
\end{equation}
\noindent {\bf{i)}} If $I=\{0,1,2,...n\}\mbox{  with  }n\geq 4$ then $G\cong F$. \\
{\bf{ii)}} If $I=\mathbf Z$ then $G\cong F^{'}$.\\
{\bf{iii)}} If $I=\mathbf N$ then $G\cong D$.
\end{theorem}
\begin{proof}
{\bf{i)}} Let $F$ be given as in Lemma \ref{lem1}. Define a map $f(x_n)=g_n$, $f(x_{n-1})=g_{n-1}g_n$,..., $f(x_0)=g_0g_1\cdots g_n$. For ${\it v}$ the corresponding map on the free group 
we check relations (\ref{fig}). We have ${\it v}(x_jx_i)=g_j\cdots g_ng_i\cdots g_n$ and \\ 
${\it v}(x_ix_{j+1})=g_i\cdots g_ng_{j+1}\cdots g_n$ for $0<i<j<n$. It all amounts now to check the following relation: $g_j\cdots g_ng_i\cdots g_j=g_i\cdots g_n$. Because of commutations (\ref{r2}) the left-hand side can be rewritten and the relation to be checked becomes 
$$g_i\cdots g_{j-2}g_jg_{j-1}g_{j+1}g_jg_{j+2}\cdots g_n=g_i\cdots g_n$$
Simplifying by $g_i\cdots g_{j-2}$ to the left and by $g_{j+2}\cdots g_n$ to the right the last equality reduces exactly to (\ref{r1}). Clearly, $(f(x_i))_{i=0}^{n}$ generate $G$, hence by the principle above there exists a surjective morphism $f:F\rightarrow G$. If $\mbox{Ker}f$ is not trivial then by Proposition \ref{pro0} we would get that $G$ is abelian (and this cannot happen as it would be implied that some $g_i$'s are the identity). In conclusion, $f$ is an isomorphism. \\
{\bf{ii)}}
We will make use of the following groups: for $a$, $b$ in $\mathbb{Z}[\frac{1}{2}]$ and $a<b$ define 
$$F(a,b):=\{f\in\mbox{ \~{F}  }|\mbox{ }f(t)=t\mbox{ if }t\notin (a,b)\}.$$
Then $(F(-k,k))_{k\in\mathbb{N}}$ is an increasing sequence of groups and by (\ref{commR}) we have\\ 
$\mbox{\~{F}}\mbox{}^{'}=\cup_{k\geq 2}F(-k-1,k+1)$.

\begin{figure}\label{f3}
\begin{center}
\input{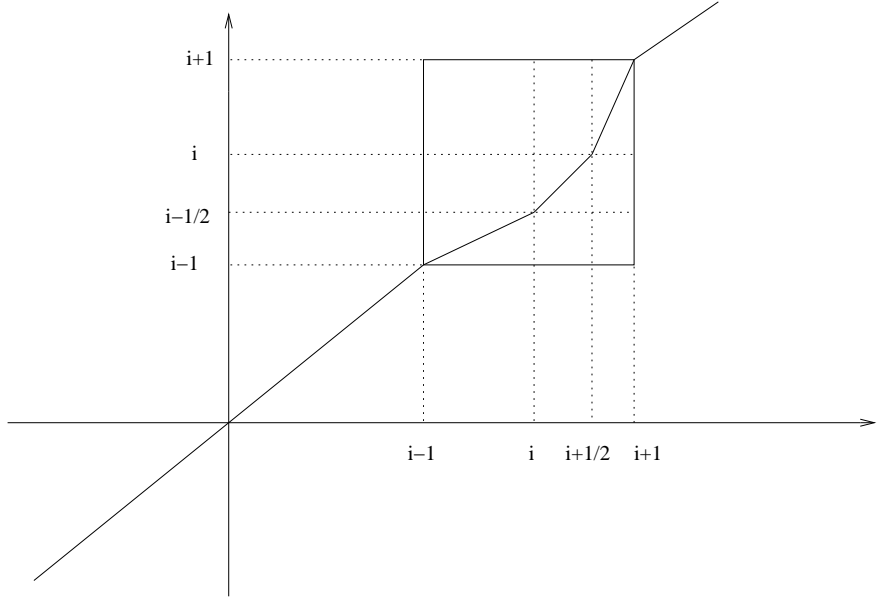}
\caption{Graph of \~{G}$_i=y_iy_{i+1}^{-1}$, $i\in\mathbb{Z}$} 

\end{center}

\end{figure}

We make the following claims:\\
$\bullet$ $F(-k-1,k+1)$ is generated by $\mbox{\~G}_{-k},\cdots \mbox{\~G}_0,\cdots \mbox{\~G}_k$;\\
$\bullet$ $\mbox{\~G}_{-k},\cdots \mbox{\~G}_0,\cdots \mbox{\~G}_k$ satisfy relations (\ref{r1}) and (\ref{r2}) and this gives a presentation of $F(-k-1,k+1)$.\\
Notice that these claims will finish the proof of ii) : the set $(\mbox{\~G}_i)_{i\in\mathbb{Z}}$ will generate \~{F}$^{'}$ and any (extra) relator, being a finite length word in letters $\mbox{\~G}_i$ will be a relator in some $F(-k-1,k+1)$. However, if extra, the relator will violate the presentation of $F(-k-1,k+1)$. By taking $2k=n$ (so that $n\geq 4$) and translating by $k$ the claims to be proved become:\\
$\bullet$ $F(-1,n+1)$ is generated by $\mbox{\~G}_0,\cdots \mbox{\~G}_{n}$;\\
$\bullet$ $\mbox{\~G}_0,\cdots \mbox{\~G}_n$ satisfy relations (\ref{r1}) and (\ref{r2}) and this gives a presentation of\\ $F(-1,n+1)$.\\
To prove the last two claims we will construct an isomorphism between (the original) $F$ and $F(-1,n+1)$ as follows: first, for any $f\in F$ we will denote by $f'$ its extension to $\mathbb{R}$ where $f'(t)=t$ outside $[0,1]$. Next we consider the sequence $s_{-1}=0$, $s_k=1-2^{-k-1}$ for $k=0,...n$ and $s_{n+1}=1$. Let $\phi_n:\mathbb{R}\rightarrow\mathbb{R}$ be the map such that $\phi_n(s_k)=k$ for all $k\in\{-1,0,1,2,...n+1\}$ with $\phi_n$ affine in any interval $[s_k,s_{k+1}]$ and $\phi_n(t)=t-1$ if $t\leq 0$, $\phi_n(t)=t+n$ if $t\geq 1$. It is not hard to check that the map 
$$F\ni f\rightarrow \phi_nf'\phi_n^{-1}\in F(-1,n+1)$$ 
is well-defined and is a group isomorphism. Let $g_k:=x_kx_{k+1}^{-1}$ for $k\in\{0,1,...n-1\}$ and $g_n=x_n$. Then $\{g_k\mbox{}|\mbox{}k=0,1,...,n\}$ generate $F$ and (\ref{r1}) and (\ref{r2}) give a presentation of $F$ (recall $n\geq 4$).
We will finish the proof once we show 
\begin{equation}\label{subcl}
 \phi_ng_k'\phi_n^{-1}=\mbox{\~G}_k\mbox{  for all }k\in\{0,1,...n\} 
\end{equation}
The equality holds outside the interval $(-1,n+1)$ because all $\mbox{\~G}_k$ with $0\leq k\leq n$ are equal to the identity map on that domain. Hence it is enough to show 
\begin{equation}\label{subcl1}
 \phi_ng_k\phi_n^{-1}(t)=\mbox{\~G}_k(t)\mbox{  for all  }t\in [-1,n+1]\mbox{,   for all }k\in\{0,1,...n\} 
\end{equation}
The case $k=n$ can be treated separately, all that is involved being calculations similar to the ones below. So let $0\leq k<n$. Using $x_k$ given in (\ref{xn}) one finds that $g_k$ is affine in between the breakpoints $s_{-1}=0$, $s_{k-1}$, $1-\frac{3}{2^{k+2}}$, $s_k$ and $s_{k+1}$. Also $x_k(u)=u$ for $u\leq s_{k-1}$, $x_k(s_{k+1})=s_{k}$ and $x_k(s_k)=1-\frac{3}{2^{k+2}}$. Now, if $t\leq k-1$ then both sides of (\ref{subcl1}) are equal to $t$. If $t\geq k+1$ then $\mbox{\~G}_k(t)=t$. Also  $x_{k+1}^{-1}(\phi_n^{-1}(t))=\frac{\phi_n^{-1}(t)+1}{2}$. Hence, using again (\ref{xn}) $g_k(\phi_n^{-1}(t))=\phi_n^{-1}(t)$ and (\ref{subcl1}) follows. It remains thus to treat the case $t\in (k-1,k+1)$. Because $\phi_n$ is affine in between $s_{k-1}$, $s_k$ and $s_{k+1}$ and $\mbox{\~G}_k$ is affine in between $k-1$,$k$, $k+\frac{1}{2}$ and $k+1$ it suffices to prove (\ref{subcl1}) for $t=k$ and $t=k+\frac{1}{2}$. Notice $\phi_n(1-\frac{3}{2^{k+2}})=k-\frac{1}{2}$. \\
For $t=k$ we have: \\ $$\phi_nx_kx_{k+1}^{-1}\phi_n^{-1}(k)=\phi_nx_kx_{k+1}^{-1}(s_k)=\phi_nx_k(s_k)=\phi_n(1-\frac{3}{2^{k+2}})$$
$$=k-\frac{1}{2}=\mbox{\~G}_k(k).$$ 
For $t=k+\frac{1}{2}$ we have:
$$\phi_nx_kx_{k+1}^{-1}\phi_n^{-1}(k+\frac{1}{2})=\phi_nx_kx_{k+1}^{-1}(1-\frac{3}{2^{k+3}})=\phi_nx_k(s_{k+1})=\phi_n(s_k)$$
$$=k=\mbox{\~G}_k(k+\frac{1}{2}).$$ 

{\bf{iii)}} Let $\phi_{\infty}:[0,1)\rightarrow [-1,\infty)$ be affine in between the points $\gamma_k=1-2^{-k-1}$ with $\phi_{\infty}(\gamma_k)=k$ for all $k\geq -1$. For any $f\in D$ define an element of $\mbox{\~{F}}$ 
$$h_f(t)=\left\{\begin{array}{lr}
\phi_{\infty}f\phi_{\infty}^{-1}(t),&\mbox{ if }t\geq -1 \\
t,&\mbox{ if }t\leq -1  
\end{array}\right.$$
Because $f$ is trivial in a neighborhood of $t=1$, $h_f$ is trivial outside an interval $[-1,n+1]$. 
The map $$D\ni f\rightarrow h_f\in \cup_{k=0}^{\infty}F(-1,k+1)$$ 
is a group isomorphism. The sequence of groups $(F(-1,k+1))_{k\geq 0}$ is increasing and by the previous proof each $F(-1,k+1)$ is generated and presented by $\mbox{\~G}_0,\cdots \mbox{\~G}_{k}$ with relations (\ref{r1}) and (\ref{r2}). It follows that $\cup_{k=0}^{\infty}F(-1,k+1)$ is generated by $\mbox{\~G}_0,\cdots \mbox{\~G}_{k},\mbox{\~G}_{k+1}\cdots$. Moreover the same relations are satisfied and this gives a presentation of the whole union (because any extra-relator would end up in some $F(-1, n+1)$).
\end{proof}

\begin{remark} Let us sketch an algebraic proof for the presentation of $F^{'}$. What follows is based on discussions with M. Brin.
Again, start with $F$ on the entire real line. We will switch the notations around a bit: the generators are $s(t)=t-1$ and 
$$x_0(t)=\left\{\begin{array}{lr}
t,&\mbox{ if }t\leq 0\\
\frac{t}{2},&\mbox{ if }0\leq t\leq 2\\
t-1,&\mbox{ if }t\geq 2
\end{array}\right.$$  
Also $x_i:=s^{-i}x_0s^i$. We define $G_i:=x_ix_{i+1}^{-1}$ for all $i\in \mathbb{Z}$.
The main point comes into play now: lemma \ref{lem1} is still valid (word for word, eventhough the 'old' $x_1$ is now called $x_0$). 
Let $H$ be the subgroup of $F$ generated by $G_i$, $i\in\mathbb{Z}$. Clearly $H$ is a subgroup of the commutator group, $F^{'}$. We can write $H$ as an increasing union of subgroups $H=\cup_{k\geq 3}H(-k,k)$ where for $n-m\geq 4$  $H(m,n)$ is by definition the subgroup of $H$ generated by $G_m$,...,$G_n$. As in part i) of theorem \ref{teo} we can apply lemma \ref{lem1} and show that $F$ is generated and presented by $G_0$, $G_1$,...$G_{n-m}$ with relations (\ref{r1}) and (\ref{r2}).
As expected $H(m,n)$ is isomorphic to $F$ and the generators $G_m$,..., $G_n$ with their corresponding relations (\ref{r1}) and (\ref{r2}) constitute a presentation of $H(m,n)$. Putting all $H(-k,k)$ together we obtain that $H$ is generated and presented by $G_i$, $i\in\mathbb{Z}$ with (\ref{r1}) and (\ref{r2}). 
\par The equality  $H=F^{'}$ will end the proof. It suffices to show $H$ is normal in $F$ or equivalently that $H$ is invariant under conjugations by $s$ and $x_0^{\pm}$. Conjugations of the $G_i$'s by $s$ only shifts subscripts so that it remains to treat conjugations by $x_0^\pm$. These are further reduced down to the following: $x_0^{-1}G_ix_0$ for $i=-1,0$ and $x_0G_ix_0^{-1}$ for $i=-1,0,1$. We only show that $x_0^{-1}G_0x_0$ is in $H$, all the other cases being reasonable to deal with. As in proof of part i) let 
$g_i=x_ix_{i+1}^{-1}$ for $i=0,1,2,3$ and $g_4=x_4$. Then $x_0=g_0g_1...g_4$, $x_1=g_1g_2...g_4$ and $(g_i)_{i=1,...4}$ satisfy relations (\ref{r1}) and (\ref{r2}). We have:
$$\begin{array}{lr} x_0^{-1}G_0x_0=x_1^{-1}x_0=g_4^{-1}g_3^{-1}g_2^{-1}g_1^{-1}g_0g_1g_2g_3g_4\\
=g_3^{-1}g_2^{-1}g_4^{-1}g_3^{-1}(g_1^{-1}g_0g_1)g_3g_4g_2g_3\\
=g_3^{-1}g_2^{-1}g_1^{-1}g_0g_1g_2g_3\\
=G_3^{-1}G_2^{-1}G_1^{-1}G_0G_1G_2G_3\in H
\end{array}$$
The third equality comes from (\ref{r1}) and the fourth from (\ref{r2}).
\end{remark}

\section{Applications}

\begin{lemma}\label{lem2}
i) For $n\in\mathbf N$ the group morphism determined by the "shift"
$$\rho_n(g_i)=g_{n+i}\mbox{,  }\forall\mbox{ }i\in\mathbf Z\mbox{,  }\forall\mbox{ }n\in\mathbf N$$
is an automorphism of $F^{'}$. \\
ii) For fixed $g$ and $h$ in $F^{'}$ there exists a large $n_0$ such that 
$$[\rho_n(g),h]=1\mbox{  for all  }n\geq n_0.$$
\end{lemma}
\begin{proof} i) One can use von Dyck's principle again to show that $\rho_n$ extends to a morphism and so  does the map defined by $\rho_{-n}(g_i)=g_{i-n}$. Clearly, these morphisms are inverse to each other. \\
ii) Write $g$ and $h$ as (finite) words in the generators $(g_i)_{i\in\mathbf Z}$ and choose $k\in\mathbf N$ such that for all $g_i$ that occur in these words $|i|\leq k$. Hence if $n\in\mathbf N$, $h$  respectively $\rho_n(g)$) are words in generators $g_i$ of index $i$ in $[-k,k]$, respectively $[n-k,n+k]$. Since $[g_i,g_j]=1$ for $|i-j|\geq 2$ it follows that  $[\rho_n(g),h]=1$, when $n\geq 2k+2$. \end{proof}

\begin{theorem}\cite{JR}\label{jr}.
The $II_1$ factor $\mathcal{L}(F')$ is asymptotically abelian.
\end{theorem}
\begin{proof} Each $(\rho_n)_n$ from Lemma \ref{lem2} extends to an inner *-automorphism of $\mathcal{L}(F')$ denoted by $\hat{\rho_n}$. We will prove that 
\begin{equation}\label{as1}
\forall x\mbox{, }y\in\mathcal{L}(F')\mbox{ :  }\lim_{n\rightarrow\infty}||[\hat{\rho_n}(x),y]||_2=0
\end{equation}
which implies that $\mathcal{L}(F')$ is asymptotically abelian. By Kaplansky's density theorem it is sufficient to prove (\ref{as1}) for $x\mbox{, }y\in\mbox{span}\{L_g\mbox{ }|\mbox{ }g\in F'\}$. From Lemma \ref{lem2} it follows that for such $x$ and $y$:
$$[\hat{\rho_n}(x),y]=0\mbox{ eventually for }n\rightarrow\infty.$$
So in particular (\ref{as1}) holds. 
\end{proof}

\begin{remark}i) In \cite{JR} Jolissaint proves a stronger result, namely that $\mathcal{L}(F')$ is inner asymptotically abelian, i.e.
$$\lim_{n\rightarrow\infty}||[\alpha_n(x),y]||_2=0\mbox{ for }x,y\in\mathcal{L}(F')$$
holds for a sequence of inner automorphisms of $\mathcal{L}(F')$. This result can also be obtained by modifying the proofs of Lemma \ref{lem2} and Theorem \ref{jr}.
\par First observe that for each $k\mbox{, }n\in\mathbf N$ there exists a $h=h_{k,n}\in F'$ such that
\begin{equation}\label{as2}
 h^{-1}g_ih=g_{i+n}\mbox{  when }|i|\leq k
\end{equation}
One can namely choose $h$ such that the corresponding element $\tilde{h}$ in \~{F} has a graph as depicted in figure 5, where $s=n-k-1$ and $ t=n+k+1$. Let $\sigma_m\in\mbox{Aut}(F')$ be the inner automorphism 
$$\sigma_m=\mbox{ ad }h_{m,2m+2}^{-1}\mbox{, }m\in\mathbf N.$$
Then it is clear from the proof of Lemma \ref{lem2} that if $g\mbox{, }h\in\mathcal{L}(F')$ are words in the generators $g_{-k}\mbox{, }g_{1-k}\cdots\mbox{, }g_k$ then 
$$[\sigma_m(g),h]=1\mbox{  for  }m\geq k.$$
Hence the proof of Theorem \ref{jr} works with $(\hat{\rho_n})_{n=1}^{\infty}$ replaced by $(\hat{\sigma_m})_{m=1}^{\infty}$, where $\hat{\sigma_m}=\mbox{ ad }(L_{h_{m,2m+2}^{-1}})$ is an inner automorphism of $\mathcal{L}(F')$ for every $m\in\mathbf N$. 
\begin{figure}\label{f4}
\begin{center}
\input{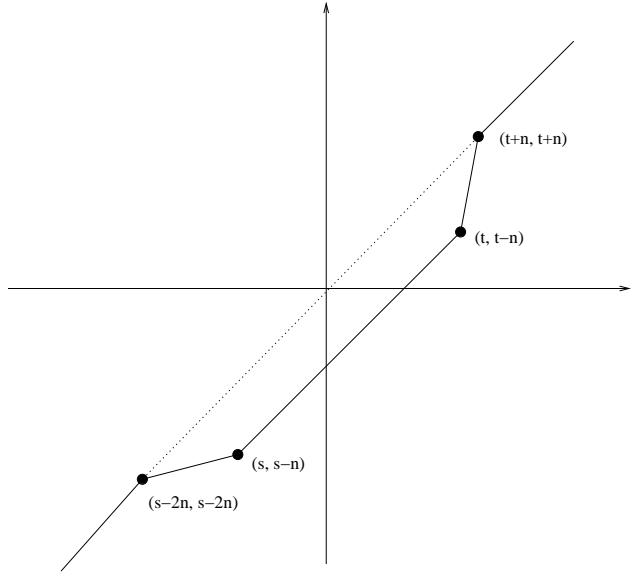}
\caption{Graph of $\tilde{h}$} 

\end{center}

\end{figure}

ii) The argument in the proof of Theorem \ref{jr} does not work if we replace $F'$ with the intermediate group $D$ or with the Thompson's group $F$. Hence we find the following question very interesting : Is $\mathcal{L}(F)$ ( or $\mathcal{L}(D)$ ) asymptotically abelian?
\end{remark}

\begin{theorem} Any non-trivial normal subgroup of $F$ has cost $1$. 
\end{theorem}
\begin{proof}
Because any proper quotient of $F$ is abelian, any non-trivial, normal subgroup of $F$ must contain $F^{'}$. Thus it suffices to show $C(F^{'})=1$, fixed price. We will write $F'$ as an increasing union of groups $G_n$ such that $G_0$ is of cost $1$, fixed price and $G_{n+1}$ is obtained out of $G_n$ and elements $g$ with the property $g^{-1}G_ng\cap G_n$ is infinite. 
Let $G_0$ be the subgroup of $F'$ generated by $(g_{2i})_{i\in\mathbf Z}$. From Theorem \ref{teo} $G_0$ is abelian, therefore its cost is $1$. Let $G_1$ be generated by $G_0$ and $g_{\pm 1}$. Because of the commutation relations we have $g_1^{-1}G_0g_1\cap G_0\supset\{g_{2i}\mbox{}|\mbox{}i<0\}$ and  $g_{-1}^{-1}G_0g_{-1}\cap G_0\supset\{g_{2i}\mbox{}|\mbox{}i>0\}$. Now it is clear how to continue: gradually add a generator $g_i$ of odd subscript to a previuos $G_n$ and use the commutation relations to insure that the set $g_i^{-1}G_ng_i\cap G_n$ is infinite. Because the even subscript generators are already in $G_0$ the $G_n$'s will exhaust the group $F'$. We can now apply theorem \ref{tga}(6) and end the proof. 
\end{proof}

\begin{definition}
A separable $C^{*}$-algebra $R$ is called residually finite dimensional (RFD) if for each non-zero $x\in R$ there exists a *-homomorphism $\pi: R \rightarrow B$ such that
$\mbox{ dim }(B) <\infty$ and $\pi(x)\neq 0$. Equivalently $R$ embeds in a $C^{*}$-algebra of the form $\prod_{n=1}^{\infty} M_{k(n)}(\mathbb{C})$ where $M_{k}(\mathbb{C})$ is the algebra of $k\times k$ matrices over the complex numbers. 
\end{definition}
We will prove that both the reduced $C^{*}$-algebra $C_r^{*}(F)$ and the full $C^{*}$-algebra
$C^{*}(F)$ associated with $F$ are not residually finite dimensional. The proof is essentially based on the fact that $F$ is not a residually finite group. However the two 'residual' notions do not compare in general. There exist residually finite groups whose reduced $C^*$-algebras are not RFD (e.g. the free non abelian group on two generators) and there exist non residually finite groups whose reduced $C^*$-algebras are RFD (e.g. $(\mathbb{Q}, +)$).

\begin{lemma}\label{nonfd}Let $A$ be a (unital) finite dimensional algebra over an arbitrary field. Then 
$F^{'}$ can not be faithfully represented in $A$. 
\end{lemma}

\begin{proof}
 
Assume that $F^{'}$ can be faithfully represented in $A$ and let $(g_i)_{i\in\mathbb{Z}}$ be our generators for $F^{'}$. For simplicity of notation we will consider $F^{'}$ as a subset of $A$. Define now:

  $$\begin{array}{lr} A_0 = A\\
  A_1 = \mbox{ the commutant of }\{g_0, g_1\}\mbox{ in }A_0\\
  A_2 = \mbox{ the commutant of }\{g_3, g_4\}\mbox{ in }A_1\\
  A_3 =\mbox{ the commutant of }\{g_6, g_7\}\mbox{ in }A_2\\
\mbox{ etc. }
\end{array}$$

Since $g_i$ and $g_j$ commute when $|i-j|\geq 2$ we have:
$$\begin{array}{lr}g_3, g_4, g_5\cdots\in A_1\\
g_6, g_7, g_8\cdots\in A_2\\
\mbox{ etc. }
\end{array}$$
But since $g_{3i}$ and $g_{3i+1}$ do not commute, $A_{i+1}$ is a proper subalgebra
of $A_i$. Hence
  $$\mbox{dim}(A_i/A_{i+1})\geq 1, i = 0,1,2,\cdots$$
which implies that $A$ is infinite dimensional.
\end{proof}

\begin{theorem}
 $C_r^{*}(F)$ and $C^*(F)$ are not RFD.
\end{theorem}
\begin{proof} We will consider $F$ as a subset of (unitary) operators in the $C^{*}$ algebra $A$, where $A$ is either $C_r^{*}(F)$ or $C^*(F)$. Suppose $A$ is RFD. Then there exists an embedding 
$\pi:A\rightarrow \prod_{n=1}^{\infty} M_{k(n)}(\mathbb{C})$. It follows that 
$\pi_{|F}:F\rightarrow \mathcal{U}(\prod M_{k(n)}(\mathbb{C}))$ is a one to one group morphism. Hence, for $g\in F^{'}$, $g\neq 1$ there exists $k$ such that $p_{k}\pi(g)\neq I_{k}$ where $p_{k}$ is the projection map onto $M_{k}(\mathbb{C})$ and $I_{k}$ is the identity matrix. We have obtained a group morphism $\psi:=p_{k}\pi_{|F}$ from $F$ to the group of invertible matrices $GL_{k}(\mathbb{C})$ which is not trivial on $F^{'}$. Because $F^{'}\cap\mbox{ Ker }\psi$ is a normal subgroup, by Proposition \ref{pro0}  $\psi$ must be one to one on $F^{'}$. This of course contradicts Lemma \ref{nonfd}. 
\end{proof}

\begin{remark} Residually finite dimensional algebras are an important class of quasidiagonal $C^{*}$-  algebras (for a detailed account of these algebras we refer the reader to \cite{Nbr}). By a theorem of Rosenberg in \cite{Ha}, if $G$ is a countable discrete group and $C_r^*(G)$ is quasidiagonal then $G$ is amenable. It is believed that the converse should also be true.  
\end{remark}

\begin{acknowledgements}We thank Matt Brin for useful discussions and for his interest in our results.
\end{acknowledgements}

\end{document}